\documentclass{article}
\usepackage[left=2.65cm,right=2.65cm,top=2.7cm,bottom=2.7cm]{geometry}
\usepackage{amsmath}
\usepackage{amsthm}
\usepackage{amscd}
\usepackage{amssymb}
\usepackage{amsfonts}
\usepackage{graphics}
\usepackage[dvipsnames]{xcolor}
\usepackage{cancel}
\usepackage{comment}

\usepackage{hyperref}
\usepackage{mathtools}
\mathtoolsset{showonlyrefs}

\hypersetup{
	colorlinks   = true, 
	urlcolor     = blue, 
	linkcolor    = blue, 
	citecolor   = red 
}

\newcommand{\boL}{\mathcal{L}}
\newcommand{\boC}{\mathcal{C}}
\newcommand{\boK}{\mathcal{K}}
\newcommand{\boI}{\mathcal{I}}

\newcommand{\R}{\mathbb{R}}
\newcommand{\boW}{\mathcal{W}}

\newcommand{\boF}{\mathcal{F}}

\newcommand{\boA}{\mathcal{A}}

\newcommand{\boP}{\mathcal{P}}
\newcommand{\boS}{\mathcal{S}}

\newcommand{\dq}{\dot{q}}
\newcommand{\dpp}{\dot{p}}
\newcommand{\AP}{\textnormal{\textbf{AP}}}

\newtheorem{theorem}{Theorem}[section]
\newtheorem{lemma}{Lemma}[section]
\newtheorem{proposition}{Proposition}[section]

\theoremstyle{definition}
\newtheorem{remark}{Remark}[section]
\newtheorem{definition}{Definition}[section]

\title{Action-minimizing periodic orbits of the Lorentz force equation with dominant vector potential}
\author{ Manuel Garz\'on\footnote{Instituto de Ciencias Matemáticas, Consejo Superior de Investigaciones Científicas, 28049 Madrid, Spain. \\E-mail: {\tt manuel.garzon@icmat.es}; ORCID: {\tt 0000-0003-4611-9181}}, Salvador L\'opez-Mart\'inez\footnote{Departamento de Matem\'aticas, Universidad Aut\'onoma de Madrid, Ciudad Universitaria de Cantoblanco, 28049, Madrid, Spain. 
		E-mail: {\tt salvador.lopez@uam.es}; ORCID: {\tt 0000-0002-5731-2385}}}

\begin{document}

\maketitle

\medskip
	
\begin{abstract}
We establish the existence of non-constant periodic solutions to the Lorentz force equation, where no scalar potential is needed to induce the electromagnetic field. Our results extend to cases where a possibly singular scalar potential is present, although the vector potential assumes a leading role. The approach is based on minimizing the action functional associated with the relativistic Lagrangian. The compactness of the minimizing sequences requires the existence of negative values for the functional, which is proven using novel ideas that exploit the sign-indefinite nature of the term involving the vector potential.
\end{abstract}

		

\section{Introduction}
The Lorentz force equation models the relativistic dynamics of a slowly accelerated charged particle in an electromagnetic field. In this context, let $q(t):[0,T]\rightarrow\R^3$ denote the position of the particle at time $t$, and let $\dq(t)$ denote its velocity. The electric and magnetic fields are represented by $E:[0,T]\times U\rightarrow\R^3$ and $B:[0,T]\times U\rightarrow\R^3$, respectively, with $U\subset\R^3$ an open set. For simplicity, and without loss of generality, we normalize both the speed of light in vacuum and the charge-to-mass ratio to one. Under these assumptions, the system under consideration takes the form
\begin{equation}\label{eq:LFE}
    \dfrac{d}{dt}\left(\frac{\dq(t)}{\sqrt{1-|\dq(t)|^2}}\right)=E(t,q(t)) + \dq(t)\times B(t,q(t)),
\end{equation}
where the right-hand side corresponds to the electromagnetic Lorentz force, and the left-hand side represents the relativistic acceleration of the particle. In this work, we are interested in periodic solutions to \eqref{eq:LFE}, that is periodic functions $q\in\boC^2(\R;U)$ with $\|\dq\|_\infty<1$, and satisfying the equation.

Equation~\eqref{eq:LFE} was independently derived in the early past century by Poincaré~\cite{Po2} and Planck~\cite{Planck}, and it stands as one of the fundamental equations in Mathematical Physics \cite{Gri, Jac}. Associated with \eqref{eq:LFE} is the Lagrangian $\boL$ defined by
\begin{equation}\label{Lagrangian}
	\boL(t,q,\dq) = 1 - \sqrt{1 - |\dq|^2} + \dq\cdot A(t,q) -\Phi(t,q),
\end{equation}
where $A:[0,T]\times U\to\R^3$ and $\Phi:[0,T]\times U\to\R$ are smooth functions forming the electromagnetic potential, i.e.,
\begin{equation}\label{Potential-Fields}
B(t,q)=\nabla\times A (t,q),\qquad E(t,q)=-\partial_t A(t,q)-\nabla\Phi(t,q).
\end{equation}
Under suitable conditions, the pair $(E,B)$ obtained from $(A,\Phi)$ solves Maxwell's equations uniquely for a specific distribution of charge and current. Here, the operators $\nabla$ and $\nabla\times$ respectively denote the gradient and the rotational with respect to the spatial variables, while $\partial_t$ stands for the time partial derivative.

From a mathematical perspective, there has been a lack of qualitative results on the dynamics of \eqref{eq:LFE} until recent years, despite its historical relevance. In particular, in 2019, Arcoya, Bereanu and Torres established in \cite{ABT} the first rigorous critical point theory for the action functional associated with \eqref{Lagrangian}. This work is based on Szulkin's variational framework \cite{Szu} for functionals having a regular part plus a lower semi-continuous term. The results in \cite{ABT} yield the existence of periodic solutions of Equation \eqref{eq:LFE} for a large class of non-singular electromagnetic fields, i.e. for those admitting a smooth extension to $[0,T]\times \R^3$ (see also \cite{Ber2}). Additionally, still in the non-singular regime, the same authors provided in \cite{ABT2} a Lusternik--Schnirelmann multiplicity theory.

Concerning the dynamics in singular fields, the existence of closed trajectories has been established for isolated singularities both by variational methods \cite{AS,Ber1,BosDamPap} and by global continuation using topological degree \cite{GT}. These works cover physically relevant models such as scalar potentials with Coulomb-type singularities, magnetic dipoles, Liénard--Wiechert potentials, among others.

In the previous literature, the existence results just outlined critically rely on the scalar potential $\Phi(t,x)$ exerting some form of control over the vector potential $A(t,x)$, regardless of the different nonlinear analysis techniques employed in these works. In particular, an everywhere zero $\Phi$ is never an admissible choice in the mentioned results. In contrast, we emphasize that any electric current through a wire induces an electromagnetic field (as a solution of Maxwell’s equations) that is completely determined by the relation \eqref{Potential-Fields} for some vector potential $A(t,x)$, with $\Phi\equiv0$. The reader may check the mathematical details in \cite{GM, Gri, Jac}. Concerning this, for the Newtonian counterpart of \eqref{eq:LFE}, a wide range of dynamical and periodic phenomena associated with wire-current configurations has been established in \cite{ALP, GM, GP1, GP2, VBB}, among others. Complementarily, \cite{GT2} is the only available work addressing the dynamics generated by such configurations in the relativistic framework of \eqref{eq:LFE}. In that paper, the authors proved the existence of solutions exhibiting radial periodic motion around an infinitely thin, straight, and infinitely long wire carrying a time-dependent current. 

In summary, the relativistic dynamics of a charged particle induced by a vector potential remains largely unexplored. With this motivation, the present paper establishes the first qualitative  results of the existence of closed trajectories in electromagnetic fields mainly driven by a vector potential. More concretely, Theorem~\ref{MainTheorem1} focuses on the case $\Phi\equiv0$ with $A(t,x)$ such that the field given by \eqref{Potential-Fields} is continuous in all variables, $T$-periodic in time, and uniformly decays to zero at infinity. From this class of vector potentials, we extend Theorem~\ref{MainTheorem1} to the general setting $\Phi\not\equiv0$ in both Theorem~\ref{MainTheorem2} and Theorem~\ref{MainTheorem3}, by establishing a control of $A$ over $\Phi$ that also includes the case $\Phi\equiv0$. Moreover, the hypotheses on $\Phi$ allow the consideration of singularities of Coulomb-type as in \cite{AS,Ber1}, while the vector potential remains smooth. 

Concerning the techniques, the solutions to Equation~\eqref{eq:LFE} are obtained as critical points of the action functional associated with the Lagrangian~\eqref{Lagrangian}, thus continuing the line of investigation initiated in~\cite{ABT}. More specifically, we identify conditions (all of them assuming a dominant role of the vector potential \( A \) relative to the scalar potential \( \Phi \)) under which the action functional admits a global minimizer. 

The key ingredients for establishing the boundedness of minimizing sequences are twofold: first, the decay of the electromagnetic field at infinity ensures that the action functional is (infinitesimally) non-negative at infinity, and second, the infimum of the functional is negative. It is worth emphasizing that, even in the simpler case \( \Phi \equiv 0 \), the existence of negative values of the functional is far from trivial, since the Lagrangian~\eqref{Lagrangian} consists of the sum of a non-negative term and an indefinite term. Our proofs of the negativity of the infimum exploit the specific structure of the first-order term in \eqref{Lagrangian} involving \( A \), while the relativistic term could be replaced  with the Newtonian operator or more general nonlinear operators. For this reason, we believe our results may be of broader relevance to minimization theory. The details of the minimization procedure are presented in Section~\ref{Section:Functional}.

We also emphasize that the solutions we obtain are all non-constant vectors of $\R^3$. Concerning this, the existence of non-constant periodic solutions for the Lorentz force equation only has been established in \cite[Theorem {10}]{ABT}, \cite[Theorem {3.7}]{AS} and \cite[Theorem 1]{GT}.

The structure of the paper is as follows. In Section~\ref{Section:Statements}, we introduce the hypotheses on the potentials and state the main results. Section~\ref{Section:Functional} is devoted to establishing a minimization result, conditional on the existence of negative values of the action functional. Section~\ref{SectionMinimum} includes the proof of Theorem~\ref{MainTheorem1}, which focuses on the minimization of the action functional in the absence of a scalar potential. In Section~\ref{Section-Phi}, we prove Theorem~\ref{MainTheorem2} and Theorem~\ref{MainTheorem3}, which extend Theorem~\ref{MainTheorem1} to more general settings where a scalar potential is present but controlled by a vector potential. In Section~\ref{Section-conclusions}, we discuss about the necessity of the assumptions made throughout the paper. Finally, the Appendix provides an uniform a priori bound for $|\dq|$ in Equation \eqref{eq:LFE} in the non-singular regime, which shows that there cannot exist a sequence of periodic solutions whose velocity approaches the speed of light in vacuum.

\section{Hypotheses and main results}\label{Section:Statements}
As previously mentioned, our interest lies in the role of the vector potential in the existence of periodic trajectories of charged particles. To this end, we begin by considering electromagnetic fields in \eqref{Potential-Fields} with $\Phi \equiv 0$, which reduces Equation \eqref{eq:LFE} to
\begin{equation}\label{eq:LFE0}
	\dfrac{d}{dt}\left(\frac{\dq(t)}{\sqrt{1-|\dq(t)|^2}}\right)=-\partial_t A(t,q(t)) + \dq(t)\times \nabla\times A(t,q(t)).
\end{equation}
After fixing $T>0$, we begin by stating our first main result, which establishes the existence of periodic solutions of \eqref{eq:LFE0} for a large class of electromagnetic fields uniquely determined by a vector potential $A\in\boC^1(\R^4;\R^3)$.
\begin{theorem}\label{MainTheorem1}
	Let $A\in\boC^1\left(\R^4;\R^3\right)$ be non-autonomous, $T$-periodic in its first variable, and with uniformly asymptotic decay at infinity, i.e.
	\begin{equation}\label{def:periodic}
		\partial_tA\not\equiv0,\quad A(0,x)=A(T,x),\ \hbox{for all }x\in\R^3,
	\end{equation}
	and
	\begin{equation}\label{cond:decayA}
		\lim_{|x|\to\infty} \big(\left|\partial_t A(t,x)\right| + |\nabla A(t,x)|\big)=0,\quad\text{uniformly in }t.
	\end{equation}
	Then the Lorentz force equation \eqref{eq:LFE0} possesses at least one non-constant periodic solution ${\bf q}$. Moreover,
	\begin{equation}
		\|\dot{{\bf q}}\|_\infty<\rho,
	\end{equation}
	for some $\rho\in(0,1)$ independent of ${\bf q}$.
\end{theorem}
To the best of our knowledge, this is the first result on the existence of closed trajectories of a charged particle in the relativistic regime of the Lorentz force equation \eqref{eq:LFE}, where the dynamics is completely independent of the scalar potential, since $\Phi$ is identically zero. We also emphasize that the solutions are non-constant, as noted previously in the Introduction. 

To simplify notation, we define the family $\boA$ of vector potentials satisfying the hypotheses of Theorem~\ref{MainTheorem1}, i.e.
\begin{equation}
    \boA:=\{A\in\boC^1(\R^4;\R^3): A \hbox{ satisfies \eqref{def:periodic} and \eqref{cond:decayA}}\}.
\end{equation}

Since our analysis also covers the general case $\Phi\not\equiv 0$, including non-smooth regimes, in what follows we fix a compact set $\boS\subset\R^3$ representing the \textsl{set of singularities} of $\Phi$, which may be empty. Moreover, we define its complement by $\boS^c:=\R^3 \setminus \boS$. 
\begin{definition}
	We define the family of scalar potentials $\boP$ as the set of functions $$\Phi:\R\times\boS^c\to\R$$that are bounded from above, $T$-periodic and continuous in the time variable, and such that 
	\begin{equation}
		\Phi(t,\cdot)\in\boC^1(\boS^c;\R),\quad\hbox{for all }t\in[0,T],
	\end{equation} and
	\begin{equation}\label{cond:decayphi}
		\lim_{|x|\to\infty}|\nabla\Phi(t,x)|=0,\ \text{uniformly in }t.
	\end{equation}
	Moreover, when $\boS\neq\emptyset$, for any point on the boundary $x_0\in\partial\boS$, and any $t_0\in[0,T]$, there exist $\varepsilon, r>0$ such that
	\begin{equation}\label{hyp:S}
		-\Phi(t,x)\geq \frac{r}{|x-x_0|},\quad\text{for all }(t,x)\in\R\times\boS^c\hbox{ with } |x-x_0|,|t-t_0|<\varepsilon.
	\end{equation}
\end{definition}
Observe that the case $\boS=\emptyset$ includes the setting of Theorem~\ref{MainTheorem1}. In this situation, the electromagnetic field in \eqref{eq:LFE0} is assumed to be continuous in $\R^4$, periodic in time, and with uniformly in time decay at infinity. In the general case, the set $\boS$ may consist of a finite union of points or, more generally, compact subsets of $\R^3$, where the singularities of the field are governed by Coulomb's law. Furthermore, the decay conditions \eqref{cond:decayA} and \eqref{cond:decayphi} are necessary to ensure uniqueness of solution to Maxwell's equations in the distributional sense \cite{GM,GT2}, thereby omitting the appearance of electromagnetic radiation phenomena \cite{Jac,Gri}. Hence, the electromagnetic fields described through $\boA\times\boP$ have a natural sense from both mathematical and physical points of view. 
 
Nevertheless, a direct computation shows that any two pairs of potentials related by the transformation
\begin{equation}\label{def:LorenzGauge}
    (\Phi_1,A_1)=(\Phi_2+\partial_t f,A_2-\nabla f),\quad\hbox{for some }f\in\mathcal{C}^1(\R^4;\R),
\end{equation}
generate the same field through \eqref{Potential-Fields}, thereby describing the same dynamics in Equation \eqref{eq:LFE}. This is the well-known \textsl{Lorenz gauge condition} of the potentials, adapted to our framework, which naturally induces the identification of a pair $(\Phi,A)$ with its equivalence class in the quotient space induced by \eqref{def:LorenzGauge}.
\begin{definition}
    We define the set $\left(\boA\times\boP\right)/\sim$ of the equivalence classes of the pairs $(A,\Phi)\in\boA\times\boP$ through the relation $\sim$, that is
\begin{equation}\label{eq:equivalenceclass1}
	(\Phi_1,A_1)\sim (\Phi_2,A_2)\Leftrightarrow(\Phi_1-\Phi_2,A_1-A_2)=(\partial_t f,-\nabla f),\quad\hbox{for some }f\in\mathcal{C}^1(\R^4;\R).
\end{equation}
\end{definition}
We clarify that, whenever we take an element in $\boA\times\boP/\sim$, we implicitly choose a representative within its equivalence class. Thus, we will abuse of the notation and simply write $(A,\Phi)\in(\boA\times\boP)/\sim$. 

Since our analysis focuses on the purely relativistic dynamics of \eqref{eq:LFE}, it is necessary to exclude magnetostatic configurations. More concretely, by Maxwell's equations, the absence of electric field yields the magnetic field to be time-independent necessarily. In particular, it is straightforward to prove that the modulus $|\dq(t)|$ is conserved in Equation \eqref{eq:LFE} when $E\equiv0$ and $B(t,q)=B(q)$. As a consequence, the Lorentz force equation \eqref{eq:LFE} reduces to a second-order Newton–Lorentz equation, thereby losing its relativistic character.

Given these clarifications, we establish the family of potential $(A,\Phi)$ under consideration.

\begin{definition}
The set $\AP$ of admissible potentials is defined as follows 
\begin{equation}\label{cond:admissible}
\AP:=\{(A,\Phi)\in\left(\boA\times\boP\right)/\sim,\quad\hbox{such that }\quad \partial_t A + \nabla \Phi\not\equiv 0\}.
	\end{equation}
\end{definition}
Recall from \eqref{Potential-Fields} that the elements in $\AP$ generate a non-zero electric field. Moreover, since every element in $(\boA\times\boP)/\sim$ gives rise to the same electric field, the set $\AP$ is well-defined. In addition, observe that the condition $\partial_t A + \nabla\Phi\not\equiv 0$ is automatically satisfied whenever $\boS\not=\emptyset$, since $A$ is smooth on $\boS$.

Finally, we stress that the elements of $\boP$ can be considered to have non-positive mean value. Specifically, for each $\Phi\in\boP$, let us define the function $\varphi:\boS^c\to\R$ as follows:
	\begin{equation}\label{def:varphi}
		\varphi(x)=\int_0^T\Phi(t,x)dt.
	\end{equation}
    Since $\Phi$ is bounded from above, it follows that the function $\varphi$ has finite supremum. Then, the potential $\hat{\Phi}$ defined by
	\begin{equation}
		\hat{\Phi}(t,x):=\Phi(t,x)-\frac1T \sup\varphi,\ \hbox{for all }(t,x)\in\R\times\boS^c, 
	\end{equation}
    belongs to $\boP$, satisfies that
    \[\int_0^T\hat\Phi(t,x)dt\leq 0,\]
    and $(A,\hat{\Phi})\sim(A,\Phi)$, for all $A\in\boA$. Therefore, from now on we shall assume without loss of generality that
	\begin{equation}\label{eq:supremum-varphi}
		\sup\limits_{x\in\boS^c}\varphi(x)=0.
	\end{equation}
    
    We emphasize on the dichotomy regarding whether the mean value function $\varphi(x)$, defined in \eqref{def:varphi}, attains its supremum \eqref{eq:supremum-varphi} or not. In fact, when $\varphi$ does not attain a maximum, it must necessarily decay to zero at infinity, i.e.
\begin{equation}\label{cond:decay-varphi}
	\varphi(x)\to0,\hbox{ as }|x|\to\infty.
\end{equation}We then state our second result, which generalizes Theorem~\ref{MainTheorem1}.

\begin{theorem}\label{MainTheorem2}
Let $(A,\Phi)\in\AP$ be such that, for some sequence $\{b_n\}\subset\R^3$ with $|b_n|\to\infty$ as $n\to\infty$, the function
\begin{equation}
	\tilde{A}(t,x):=A(t,x)-\dfrac{1}{T}\int_{0}^{T}A(t,x)dt,
\end{equation}
satisfies the following conditions:
\begin{equation}\label{cond:varphi0}
    \partial_t A(\cdot,b_n)\not\equiv 0,\quad\hbox{for all }n\in\mathbb{N},
\end{equation}
\begin{equation}\label{cond:varphi1}
	    \lim\limits_{n\to\infty}\varphi(b_n)\|\tilde A(\cdot,b_n)\|^{-2}_2=\lim\limits_{n\to\infty}\varphi(b_n)\left(\int_0^T\left|\tilde A(t,b_n)\right|^2dt\right)^{-1}=0,
\end{equation}
and
\begin{equation}\label{cond:varphi2}
	\lim_{n\to\infty}\max\{|\nabla\Phi(t,y)|:\,t\in[0,T],\,|y|\geq |b_n|-T\}\|\tilde A(\cdot,b_n)\|^{-1}_2=0.
\end{equation}
Then, the Lorentz force equation \eqref{eq:LFE} possesses at least one non-constant periodic solution ${\bf q}$. Moreover, when $\boS=\emptyset$,
\begin{equation}
	\|\dot{{\bf q}}\|_\infty<\rho,
\end{equation}
for some $\rho\in(0,1)$ independent of ${\bf q}$.
\end{theorem}

Intuitively, conditions \eqref{cond:varphi1} and \eqref{cond:varphi2} require, at least in one direction, the decay of both the mean value and the gradient of $\Phi(t,x)$ to be faster than the  decay of the oscillations of $A(t,x)$. These hypotheses represent the first instance in the literature, to our knowledge, where the scalar potential $\Phi$ is effectively controlled by the vector field $A$, revealing the novelty of our approach.

On the other hand, note that if the decay condition \eqref{cond:decay-varphi} does not hold, then the function $\varphi$ attains its maximum. In this setting, let $\Phi\in\boP$ and $b\in\boS^c$ be such that
\begin{equation}\label{def:varphi-maximum}
	\varphi(b)=\int_0^T\Phi(t,b)dt=0=\max\limits_{x\in\R^3}\varphi(x),
\end{equation}
and let $M_\Phi\subset\boS^c$ be the set of points satisfying \eqref{def:varphi-maximum}, i.e.
\begin{equation}
M_\Phi:=\{b\in\boS^c:\ \varphi(b)=\max\limits_{\boS^c} \varphi\}.
\end{equation}

The following result addresses a complementary regime to that considered in Theorem~\ref{MainTheorem2}. While the scope of these results overlap in some cases, they also clearly cover distinct situations.

\begin{theorem}\label{MainTheorem3}
	Let $(A,\Phi)\in\AP$ be such that there exists a point $b\in M_\Phi$ satisfying that
    \begin{equation}\label{cond:nonequilibrium-ElecField}
			\partial_t A(t,b)\not\equiv-\nabla\Phi(t,b).
		\end{equation}	
    Then, the Lorentz force equation \eqref{eq:LFE} possesses at least one non-constant periodic solution ${\bf q}$. Moreover, when $\boS=\emptyset$,
\begin{equation}
	\|\dot{{\bf q}}\|_\infty<\rho,
\end{equation}
for some $\rho\in(0,1)$ independent of ${\bf q}$.
\end{theorem}

We conclude the section with some remarks analyzing notable cases where Theorems~\ref{MainTheorem2} and~\ref{MainTheorem3} are applicable.

\begin{remark}\label{Rem-Compact support}

Suppose that $\tilde A$ has compact support in $x$, independent of $t$, say $K \subset \mathbb{R}^3$. Also assume that there exists another compact set $K_0 \subset\subset K$ such that
\begin{equation}\label{eq:compactsupports}
\varphi(x) = 0, \quad\hbox{and}\quad \max\{|\nabla\Phi(t, y)| : t \in [0, T],\, |y| \geq |x| - T\} = 0,\quad\hbox{for all }x\in\R^3\setminus K_0.
\end{equation}
This situation can be interpreted as a limiting case of Theorem~\ref{MainTheorem2}, where $A$ still governs the behavior of $\Phi$, since, roughly speaking, $\Phi$ must vanish at infinity before $\tilde A$ does. However, Theorem~\ref{MainTheorem2} is not directly applicable, because the zero sets of $\varphi$ and $\tilde A$ intersect at infinity, so conditions~\eqref{cond:varphi1} and~\eqref{cond:varphi2} are ill-defined. Nevertheless, from the condition on $\nabla\Phi$ in~\eqref{eq:compactsupports}, it follows that $\nabla\Phi(t, x) = 0$ for all $x \in \mathbb{R}^3 \setminus K_0$. On the other hand, by the definition of support, we have $\partial_t A(\cdot, x) \not\equiv 0$ for all $x \in K$. This implies that
\[
\nabla\Phi(\cdot, x) \equiv 0 \not\equiv \partial_t A(\cdot, x), \quad \text{for all } x \in K \setminus K_0.
\]
At the same time, $\varphi(x) = 0$ for all $x \in K \setminus K_0$. Therefore, Theorem~\ref{MainTheorem3} applies in this case.
\end{remark}

\begin{remark}

A neat particular case that combines both Theorem~\ref{MainTheorem2} and Theorem~\ref{MainTheorem3} is the case of $\Phi$ being autonomous. Indeed, let $(A,\Phi)\in\AP$ be such that $\Phi$ does not depend on $t$, so that $\varphi(x)=T\Phi(x)\leq 0$. Thus, if there exists $b\in\boS^c$ such that 
\[\Phi(b)=0,\quad \partial_t A(\cdot,b)\not\equiv 0,\]
then $\nabla\Phi(b)=0\not\equiv\partial_t A(\cdot,x)$, and Theorem~\ref{MainTheorem3} yields the existence of a non-constant periodic solution to \eqref{eq:LFE}. If, on the contrary, $\Phi(x)<0$ for every $x\in\boS^c$, assuming in addition that $\partial_t A(\cdot,x)\not\equiv 0$ for every $|x|\gg 1$, and
        \[\lim_{|x|\to\infty}\frac{\Phi(x)}{\|\tilde A(\cdot,x)\|_2^2}=0,\quad\lim_{|x|\to\infty}\frac{\max_{\{|y|\geq|x|-T\}}|\nabla\Phi(y)|}{\|\tilde A(\cdot,x)\|_2}=0,\]
we deduce from Theorem~\ref{MainTheorem2} that, again, a non-constant periodic solution to \eqref{eq:LFE} exists.
\end{remark}

\section{Minimization of the action functional}\label{Section:Functional}

The purpose of this section is to prove, under natural conditions, the existence of a non-constant minimizer of the action functional associated to the relativistic Lagrangian \eqref{Lagrangian}. In the language of Classical Mechanics, this amounts to proving a principle of least action. The results discussed in the first two subsections are essentially established in \cite{ABT,AS}, we include them here  to keep the exposition self-contained. In contrast, the third subsection constitutes the novelty of our approach, see Remark~\ref{remark:PLA}. 

\subsection{Functional framework}

Let $\boW$ be the following Banach space 
\begin{align*}
    \boW=\{q:\R\to\R^3,\text{ Lipschitz and }T\text{-periodic}\}
\end{align*}
endowed with the norm
\begin{equation}\label{norm}
\|q\|_\boW:=\|q\|_{1,\infty}=\|q\|_\infty+\|\dq\|_{\infty},\ \hbox{for all }q\in\boW.
\end{equation}
Consider also the subset 
\begin{equation*}
    \boK=\{q\in\boW:\, \|\dq\|_\infty\leq 1\},
\end{equation*}
which is convex and closed in $\boW$ \cite[Lemma 3.1]{ABT}.

Throughout the paper, we will frequently make use of the standard decomposition $\boW = \tilde{\boW} \oplus \R^3$, where
\begin{equation}\label{def:tildeW}
    \tilde\boW=\left\{q\in\boW:\, \bar q=0\right\},\quad\hbox{with}\quad    \bar q=\frac{1}{T}\int_0^Tq(t)dt,
\end{equation}
so that every $q\in\mathcal{W}$ can be uniquely written as 
\begin{equation}\label{eq:decomposition}
    q= \tilde q + \bar q.   
\end{equation}
This framework allows to use both the classical Poincar\'e-Wirtinger inequality:
\begin{equation}\label{eq:poincare}
	\frac{\pi^2}{T^2}\|\tilde q\|^2_2\leq\|\dq\|^2_2,\quad\text{for every }q\in\boW,
\end{equation}
and the estimate
\begin{equation}\label{eq:qtildebounded}
    \|\tilde q\|_\infty\leq T,\quad\text{for all }q\in\boK,
\end{equation}
which follows directly from 
the (integral version of) the mean value theorem.

For any $\Phi\in\boP$, let us define the \textsl{non-singular set}
\begin{equation}\label{def:Lambda}
	\Lambda=\left\{q\in\boW:\ \varphi(q)>-\infty\right\},\quad\hbox{and}\quad\boK_\Lambda=\boK\cap\Lambda,
\end{equation} 
where we have considered the obvious extension of the non-positive function $\varphi$, defined in \eqref{def:varphi}, to the whole space $\boW$. In particular, $\Lambda=\boW$ when $\boS=\emptyset$. Moreover, if $\boS\not=\emptyset$, we note that every $q\in\boW$ for which there exists $t_0\in \R$ such that $q(t_0)\in\boS$ satisfies, by condition \eqref{hyp:S}, that
\[-\Phi(t,q(t))\geq \frac{c}{|q(t)-q(t_0)|},\quad\text{for all }t\in [t_0-\delta,t_0+\delta],\]
for some $\delta>0$ small enough. This inequality, together with the fact that $q$ is Lipschitz, lead to the conclusion that the non-singular set \eqref{def:Lambda} can equivalently be written as
\[\Lambda=\left\{q\in\boW:\ q(t)\not\in\boS,\text{ for all }t\in\R\right\},\]
as defined in \cite{AS} for $\boS=\{0\}$. This shows that $\Lambda$ does not depend on $\Phi$. In any case, it is simple to check that $\Lambda$ is an open subset of $\boW$.

In this setting, the action functional $\boI:\Lambda\to(-\infty,\infty]$ associated to \eqref{eq:LFE} is defined by
\begin{equation}\label{eq:functional}
    \boI(q)=\left\{\begin{array}{ll}
    	\displaystyle\int_0^T \left(1-\sqrt{1-|\dq|^2}\right)dt + \displaystyle\int_0^T \left(\dq\cdot A(t,q)-\Phi(t,q)\right)dt,  &\text{ if }q\in\boK_\Lambda,
        
        \vspace{0.3 cm}
        \\
        
    	\infty,\quad &\text{ if }q\in\Lambda\setminus\boK_\Lambda.
    \end{array}\right.
\end{equation}
It is standard to decompose the functional as $\boI = \Psi + \mathcal{F}$, where
 \begin{equation}\label{eq:Psi}
    \Psi(q)=\int_0^T \left(1-\sqrt{1-|\dq|^2}\right)dt,\ \text{if }q\in\boK, \quad \Psi(q)=\infty,\ \text{if } q\in\boW\setminus\boK,
\end{equation}
and
\begin{equation}\label{eq:F}
    \boF(q)=\int_0^T \left(\dq\cdot A(t,q)-\Phi(t,q)\right)dt,\quad \hbox{for all  }q\in\Lambda.
\end{equation}
It is straightforward to verify that the functional $\boF$ is of class $\mathcal{C}^1$. Moreover, it is well-known that $\Psi:\boW \to (-\infty,\infty]$ is a convex, proper functional with closed domain $\boK \subset \boW$. In particular, $\Psi$ is continuous on $\boK$ and lower semi-continuous on the whole space $\boW$.

\begin{remark}
    As commented in the Introduction, the structure of $\boI$ is a particular instance of Szulkin's functionals \cite{Szu}. The situation $\boS=\emptyset$ was rigorously established in \cite{ABT}, corresponding to the setting of continuous electromagnetic fields. Furthermore, the case $\boS=\{0\}$ is analyzed in \cite{AS,Ber1}, which describes an electric field with an isolated Coulomb-type singularity and a continuous magnetic field. This case was later generalized in \cite{BosDamPap} for singularities depending on the time variable, aiming to model the dynamics induced by moving particles through the Liénard-Wiechert potentials. In our work, we choose not to extend the analysis to traveling point singularities, for the sake of simplicity. However, we believe that our results should remain valid in that framework, provided the moving singularities are controlled by a time-dependent version of Coulomb’s law analogous to \eqref{hyp:S}.
\end{remark}
In this non-smooth setting, critical points are defined as follows.
\begin{definition}
	A point $q\in \boK_\Lambda$ is said to be critical for $\boI$ if it satisfies the following inequality:
	\begin{equation}\label{CriticalPoint}
		\Psi(p)-\Psi(q)+\mathcal{F}'(q)[p-q]\geq 0,\quad \hbox{for all }p\in \boK_\Lambda.
	\end{equation}
\end{definition}\noindent
An argument analogous to that of \cite[Theorem 6]{ABT} shows that a function $q\in\boW$ solves \eqref{eq:LFE} if and only if $q$ is a critical point for the functional \eqref{eq:functional}. Moreover, it follows directly from the convexity of $\Psi$ that every local minimizer of $\boI$ is necessarily a critical point of the functional, see \cite[Proposition 1.1]{Szu}.

Critical points that are not minimizers are typically obtained via min-max variational methods. In the literature of min-max theory, establishing the compactness of Palais--Smale sequences is usually a challenging task. In fact, proving a suitable compactness of Palais--Smale sequences (in the Szulkin's sense) for min-max results applied to the Lorentz force equation was one of the main contributions of \cite{ABT}. In the present paper, however, our approach is based on direct minimization, and thus we avoid dealing with Palais--Smale sequences or min-max methods. Nonetheless, compactness issues still arise in our framework.

To illustrate this, consider the case where $\boS = \emptyset$, and take the functional \eqref{eq:functional} with $A \equiv 0$, that is,
\[
\boI(q) = \Psi(q) - \varphi(q).
\]
Since $\varphi \leq 0$, it follows that $\boI(q) \geq 0$ for all $q \in \boW$. Now, suppose further that $\lim_{|x| \to \infty} \varphi(x) = 0$. Then, for any sequence $\{x_n\} \subset \mathbb{R}^3$ with $|x_n| \to \infty$, we have $\boI(x_n) \to 0$. Hence, there exist minimizing sequences that do not converge. We will show that introducing a non-autonomous vector potential $A$ provides the desired compactness.

\subsection{Semicontinuity property of the action functional}

The following result establishes an intrinsic compactness property of the set $\boK$. 
\begin{lemma}\label{lemma:compactness}
    Let $\{q_n\}\subset\boK$ be bounded in $L^\infty$. Then, there exists $q\in\boK$ such that, up to a subsequence,
    \begin{align}
        &q_n\to q,\quad\text{strongly in }L^\infty,\label{eq:Linftyconv}
        \\
        &\dot{q}_n\rightharpoonup \dot{q},\quad\text{in the weak-$^*$ topology }\sigma(L^\infty,L^1).\label{eq:weak*conv}
    \end{align}
\end{lemma}

\begin{proof}
    First, the sequence $\{q_n\}$ is bounded in $W^{1,\infty}$, and hence also bounded in the Sobolev space $H^1$. Consequently, there exists $q\in H^1$ such that, up to a not relabeled subsequence,
    \begin{equation}
        q_n\rightharpoonup q,\quad\text{weakly in }H^1. \label{eq:H1conv}
    \end{equation}
    Moreover, by the compact embedding $W^{1,\infty} \hookrightarrow \boC$, there exists $p \in \boC$ such that, up to a further subsequence,
    \begin{equation}\label{infconv}
        \|q_n-p\|_\infty\to 0.
    \end{equation}
    In particular, this implies pointwise convergence, and hence $p$ is $T$-periodic. Furthermore,  $\|q_n-p\|_2\to 0$, so that $q_n\rightharpoonup p$ weakly in $L^2$. In view of \eqref{eq:H1conv}, we deduce that $p=q$, which establishes \eqref{eq:Linftyconv}.\\
    Now we apply the Banach--Alaoglu--Bourbaki theorem. Since $\|\dot{q}_n\|_\infty\leq 1$ for all $n$, there exists $\hat{p}\in L^\infty$ such that, up to a further subsequence,
    \[\dot{q}_n\rightharpoonup \hat p,\quad\text{ in the weak-$^*$ topology }\sigma(L^\infty,L^1).\]
    This trivially implies that $\dot{q}_n\rightharpoonup \hat p$ weakly in $L^2$. Recalling \eqref{eq:H1conv}, we conclude that $\hat p=\dot{q}$, which proves \eqref{eq:weak*conv}. Finally, since the $L^\infty$-norm is lower semicontinuous with respect to the convergence \eqref{eq:weak*conv}, it follows that $\|\dot{q}\|_\infty\leq 1$, and thus $q\in\boK$.
\end{proof}

Let us recall that the functional $\Psi:\boW\to\R$ defined in \eqref{eq:Psi} is lower semicontinuous with respect to the $L^\infty$ norm \cite[Proposition 1]{BerJebMaw}, and consequently also with respect to the $W^{1,\infty}$ norm. Moreover, the functional $\boF$, defined in \eqref{eq:F}, is of class $\boC^1$ with respect to the $W^{1,\infty}$ norm. The following result provides a kind of lower semicontinuity property of $\boI$ with respect to the convergences \eqref{eq:Linftyconv} and \eqref{eq:weak*conv}.

\begin{lemma}\label{lemma:wlsc}
    Let $\{q_n\}\subset \boK_\Lambda$ be a sequence bounded in $L^\infty$, and let $q\in\boK_\Lambda$. Assume that \eqref{eq:Linftyconv} and \eqref{eq:weak*conv} hold. Then, for any compact set $\boS\subset\R^3$ and any pair $(A,\Phi)\in\AP$, the following holds:
    \begin{equation}\label{eq:Ilsc}
        \boF(q)=\lim_{n\to\infty}\boF(q_n),\quad \boI(q)\leq\liminf_{n\to\infty}\boI(q_n).
    \end{equation}
\end{lemma}

\begin{proof}
    First, using regularity of $A$ together with the convergences \eqref{eq:Linftyconv} and \eqref{eq:weak*conv}, we obtain
    \begin{equation}\label{eq:Aconvergence}
        \int_0^T \dot{q}_n\cdot A(t,q_n)dt\to \int_0^T \dot{q}\cdot A(t,q)dt.
    \end{equation}
    Moreover, since $\Phi(t,\cdot)\in\mathcal{C}^1(\boS^c;\R)$, \eqref{eq:Linftyconv} implies
    \[\int_0^T\Phi(t,q_n)dt\to \int_0^T\Phi(t,q)dt,\]
    which establishes
    \begin{equation*}
    	\lim_{n\to\infty}\boF(q_n)=\boF(q).
    \end{equation*}
    Finally, the lower semicontinuity of $\Psi$ with respect to the $L^\infty$-norm yields
    \[\boI(q)\leq\liminf_{n\to\infty}\boI(q_n),\] thus proving \eqref{eq:Ilsc}.
\end{proof}

Motivated by the previous lemma, one may ask whether every sequence in $\boK_\Lambda$ that converges in the sense of \eqref{eq:Linftyconv}-\eqref{eq:weak*conv} has its limit also lying in $\boK_\Lambda$. This is indeed the case, provided the action functional remains bounded along the sequence. This idea originates from \cite[Lemma 3.2]{AS}, even though our result is valid for slightly more general singular sets. The precise statement is the following.

\begin{lemma}\label{lemma:KLambdaclosed}
    Let $\{q_n\}\subset \boK_\Lambda$ be a sequence satisfying \eqref{eq:Linftyconv}-\eqref{eq:weak*conv} for some $q\in\boK$. Let $\boS\subset\R^3$ be a compact set, and let $(A,\Phi)\in\AP$. Assume that $\{\boI(q_n)\}$ is bounded. Then, $q\in\boK_\Lambda$.
\end{lemma}	

\begin{proof}
    Recall that \eqref{eq:Linftyconv}-\eqref{eq:weak*conv} are enough to assure \eqref{eq:Aconvergence}. Then, given that $\{\boI(q_n)\}$ is bounded, we deduce
		\[-\varphi(q_n)=-\int_0^T\Phi(t,q_n(t))dt=\boI(q_n)-\Psi(q_n)-\int_0^T \dq_n\cdot A(t,q_n(t))dt\leq C,\]
    for some constant $C>0$. Thus, since $-\Phi\geq 0$, we may apply Fatou lemma, leading to
		\[-\varphi(q)=-\int_0^T\Phi(t,q(t))dt\leq\liminf_{n\to\infty}\int_0^T(-\Phi(t,q_n(t)))dt\leq C,\]
		where $-\Phi(t,q(t))=\liminf_{n\to\infty}(-\Phi(t,q_n(t)))$. Therefore, $\varphi(q)>-\infty$, and thus,  $q\in\boK_\Lambda$ as defined in \eqref{def:Lambda}.
\end{proof}

\subsection{Compactness of minimizing sequences}

It is remarkable that the decay condition \eqref{cond:decayA} ensures that the term involving the vector potential in $\boI$  vanishes at infinity, even if $A$ itself does not (actually, $A$ could be growing at logarithmic rate). This statement is made rigorous in the following lemma.

\begin{lemma}\label{lemma:vanishing}
    Let $A\in\boA$, and let $\{q_n\}\subset\boK$ be a sequence with $\lim\limits_{n\to\infty}|\bar q_n|=\infty$. Then,
    \begin{equation}
         \lim_{n\to\infty}\int_0^T \dq_n\cdot A(t,q_n)dt= 0.\label{statement:vanishingA}
    \end{equation}
\end{lemma}
\begin{proof}

Let us write
\begin{equation}\label{Adecomposed}
    \int_0^T\dq_n\cdot A(t,q_n)dt=\int_0^T \dq_n\cdot (A(t,q_n)-A(t,\bar q_n))dt + \int_0^T\dq_n\cdot A(t,\bar q_n)dt.
\end{equation}
We start by estimating the first term on the right-hand side of \eqref{Adecomposed}. To do so, note that for every $t\in [0,T]$ and every $n$, there exists $s_n(t)\in [0,1]$ such that
\begin{equation*}
    |A(t,q_n(t))-A(t,\bar q_n)|\leq |\nabla A(t,s_n(t)q_n(t)+(1-s_n(t))\bar q_n)||q_n(t)-\bar q_n|.
\end{equation*}
Then,
\begin{equation}\label{firstterm}
    \left|\int_0^T \dq_n\cdot (A(t,q_n)-A(t,\bar q_n))dt\right|\leq \int_0^T|\nabla A(t,s_n(t)q_n(t)+(1-s_n(t))\bar q_n)||\tilde q_n|dt.
\end{equation}
Hence, bearing \eqref{eq:qtildebounded} in mind, and observing that
\[|s_n(t)q_n(t)+(1-s_n(t))\bar q_n)|\geq |\bar q_n|-|\tilde q_n|\geq |\bar q_n|-T,\]
we conclude from \eqref{firstterm} and \eqref{cond:decayA} that
\begin{equation}
    \lim_{n\to\infty}\left|\int_0^T \dq_n\cdot (A(t,q_n)-A(t,\bar q_n))dt\right|=0.
\end{equation}

Concerning the second term in \eqref{Adecomposed}, integrating by parts yields
\begin{equation*}
    \int_0^T \dq_n\cdot A(t,\bar q_n)dt=\int_0^T \frac{d}{dt}(q_n-\bar q_n)\cdot A(t,\bar q_n)dt=-\int_0^T\tilde q_n\cdot \partial_t A(t,\bar q_n)dt.
\end{equation*}
Therefore, again \eqref{eq:qtildebounded} and \eqref{cond:decayA} lead to
\begin{equation*}
    \lim_{n\to\infty}\left|\int_0^T \dq_n\cdot A(t,\bar q_n)dt\right|=0.
\end{equation*}
The proof is complete.
\end{proof}

In the line of Lemma~\ref{lemma:vanishing}, next result shows a vanishing property of the difference $\varphi(q_n)-\varphi(\bar q_n)$ at infinity.

\begin{lemma}\label{lemma:varphivanishing}
    Let $\Phi\in\boP$, and let $\{q_n\}\subset\boK_\Lambda$ be a sequence with $\lim\limits_{n\to\infty}|\bar q_n|=\infty$. Then, 
    \begin{equation}
         \lim_{n\to\infty}\int_0^T \left(\Phi(t,q_n)-\Phi(t,\bar q_n)\right)dt= 0.\label{statement:vanishingPhi}
    \end{equation}
\end{lemma}

\begin{proof}
    For any $t\in [0,T]$ and any $n$, consider a point in the segment joining $q_n(t)$ with $\bar q_n$, i.e. $sq_n(t)+(1-s)\bar q_n$ for some $s\in [0,1]$. Then, as in the proof of Lemma~\ref{lemma:vanishing}, we have
    \begin{equation}\label{segment}
    |s q_n(t)+(1-s)\bar q_n|\geq |\bar q_n|-T.
    \end{equation}
    This means that, for $n$ large enough independent of $t$, any point in the segment belongs to $\boS^c$, since $\boS$ is bounded. Therefore, the mean value theorem can be applied to obtain some $s_n(t)\in [0,1]$ such that
    \[\Phi(t,q_n(t))-\Phi(t,\bar q_n)=\nabla\Phi(t,s_n(t)q_n(t)+(1-s_n(t))\bar q_n)\cdot \tilde q_n(t).\]
    Hence,
    \[\left|\int_0^T \left(\Phi(t,q_n)-\Phi(t,\bar q_n)\right)dt\right|\leq T\int_0^T|\nabla\Phi(t,s_n(t)q_n(t)+(1-s_n(t))\bar q_n)|dt.\]
    Finally, \eqref{segment} and condition \eqref{cond:decayphi} yield the result.
\end{proof}

A direct consequence of Lemma~\ref{lemma:vanishing} and Lemma~\ref{lemma:varphivanishing} is that $\boI$ is bounded from below, as stated by the following result. 

\begin{lemma}\label{lemma:Ibounded}
    For any $(A,\Phi)\in\AP$, the action functional $\boI:\Lambda\to\R$ is bounded from below.
\end{lemma}

\begin{proof}
Given $\varepsilon>0$,  Lemma~\ref{lemma:vanishing} and Lemma~\ref{lemma:varphivanishing}, together with the fact that $\varphi$ is non-positive at the constants, yield the existence of $R>0$ such that
\begin{equation*}
\boI(q)=\Psi(q)+\int_0^T\dq\cdot A(t,q)dt - \int_0^T\left(\Phi(t,q)-\Phi(t,\bar q)\right)dt -\varphi(\bar q)\geq -\varepsilon,
\end{equation*}
for all $q\in\boK_\Lambda$ with $|\bar q|>R$. On the other hand, if $q\in\boK_\Lambda$ satisfies $|\bar q|\leq R$, then  \eqref{eq:qtildebounded} implies that $\|q\|_\infty\leq T+R$, and therefore,
\[\boI(q)\geq -T\max\{|A(t,x)|:\ t\in [0,T],|x|\leq T+R\},\]
for all $q\in\boK_\Lambda$ such that $|\bar q|\leq R$. In any case, $\boI$ is bounded from below.
\end{proof}

The previous lemma yields the existence of minimizing sequences in $\boK_\Lambda$. In view of the semi-continuity property of the action proved in Lemma~\ref{lemma:wlsc}, it only remains to prove a suitable compactness result for the minimizing sequences in order to apply the direct method of the calculus of variations.

We finish the section by stating and proving a principle of least action which ensures the compactness of the minimizing sequences provided that the action attains negative values. 

\begin{theorem}\label{thm:PLA}
    Let $(A,\Phi)\in\AP$ satisfy that $\inf\limits_{q\in\Lambda}\boI(q)<0$. Then, there exists a non-constant ${\bf q}\in\boK_\Lambda$ such that $\boI({\bf q})=\min\limits_{q\in\Lambda}\boI(q)$.
\end{theorem}

\begin{remark}\label{remark:PLA}
Let us highlight that \cite[Theorem 7]{ABT} establishes a similar principle of least action for the action functional in the setting of continuous electromagnetic fields. Nevertheless, our approach extends this principle by allowing the presence of singularities in the scalar potential (hence in the electromagnetic field) and by identifying the key condition of a negative infimum of the action. We also stress that the negativity of the infimum not only  ensures the compactness of the minimizing sequences, but also guarantees that all critical points are non-constant, since $\boI(q)\geq 0$ for every $q\in\R^3$.

As discussed in the Introduction, proving the existence of non-constant periodic solutions to the Lorentz force equation has been a challenging task in previous works, with only a few results available on the subject.
\end{remark}

\begin{proof}[Proof of Theorem~\ref{thm:PLA}]
    By Lemma~\ref{lemma:Ibounded}, we can consider a minimizing sequence $\{q_n\}\subset\boK_\Lambda$ such that
    \begin{equation}
        \boI(q_n)\to c:=\inf\limits_{q\in\Lambda}\boI(q)<0.
    \end{equation}    
     Moreover, the non-negativity of $-\left.\varphi\right|_{\R^3}$ and $\Psi$, together with Lemma~\ref{lemma:vanishing} and Lemma~\ref{lemma:varphivanishing}, trivially lead to
\begin{equation}
	\lim\limits_{n\to\infty}\boI(p_n)=\lim\limits_{n\to\infty}(\Psi(p_n)-\varphi(\bar p_n))\geq0>c,\quad\hbox{for all }\{p_n\}\subset\boK\hbox{ with }\lim\limits_{n\to\infty}\|p_n\|_\infty=\infty.
\end{equation}
Therefore, the sequence $\|q_n\|_\infty$ is uniformly bounded, and thus satisfies the convergences from \eqref{eq:Linftyconv} and \eqref{eq:weak*conv} to some ${\bf q}\in\boK$. In fact, by Lemma~\ref{lemma:KLambdaclosed}, we have that ${\bf q}\in\boK_\Lambda$. Finally, by Lemma~\ref{lemma:wlsc}, the lower semi-continuity of $\boI$ under these convergences gives that
\begin{equation}
    \boI({\bf q})\leq\lim\limits_{n\to\infty}\boI(q_n)=\inf\limits_{q\in\Lambda}\boI(q)\leq\boI({\bf q}).
\end{equation}
This concludes the proof.
\end{proof}

\section{Existence of minimizer when \texorpdfstring{$\Phi\equiv 0$}{Phi = 0}}\label{SectionMinimum}

In this section we prove Theorem~\ref{MainTheorem1}, which corresponds with the particular case of continuous electromagnetic fields in the absence of scalar potential in \eqref{Potential-Fields}.

As previously commented, the set $\Lambda$ defined in \eqref{def:Lambda} trivially coincides with the whole space $\boW$ when $\Phi\equiv0$. Therefore, the action functional $\boI_0 : \boW \to (-\infty, \infty]$ associated with Equation~\eqref{eq:LFE0} is defined as follows:
\begin{equation}\label{eq:functional-ZeroPhi}
    \boI_0(q)=\left\{\begin{array}{ll}
    	\displaystyle\int_0^T \left(1-\sqrt{1-|\dq|^2}\right)dt + \displaystyle\int_0^T \dq\cdot A(t,q)dt, \ &\text{ if }q\in\boK,\vspace{0.3cm}\\
    	\infty, &\text{ if }q\in\boW\setminus\boK,
    \end{array}\right.
\end{equation}
where the smooth part \eqref{eq:F} of $\boI_0$ becomes now

 \begin{align}\label{eq:F0}
    \boF_0(q)=\int_0^T \dq\cdot A(t,q) dt,\ \hbox{ for all }q\in\boW.
\end{align}
In this setting, we show that the functional $\boI_0$ attains negative values whenever the vector potential $A(t,x)$ belongs to the class $\boA$. By Theorem~\ref{thm:PLA}, this guarantees the existence of a non-constant ${\bf q}\in\boK$ that is a global minimize of $\boI_0$, and therefore a critical point. As a result, since ${\bf q}\in\boK$ is a periodic solution of \eqref{eq:LFE0}, the uniform bound $\|\dot{\bf q}\|_\infty<\rho<1$ is a direct consequence of Lemma~\ref{lemma:uniformbound} in the Appendix. After this argument, Theorem~\ref{MainTheorem1} becomes an immediate consequence of the next key lemma on the existence of negative values.
\begin{lemma}\label{lemma:negative}
For any $A\in\boA$ there exists a point $p\in\boK$ such that $\boI_0(p)<0$.
\end{lemma}
\begin{proof}
We begin by recalling that $A$ is globally Lipschitz continuous in $x$, uniformly in $t$. Specifically, given $A\in\boA$ there exists a finite number $M>0$ such that
\begin{equation}\label{eq:Lip}
	|A(t,x)-A(t,y)|\leq M |x-y|,\quad\text{ for every }t\in [0,T],\, x,y\in\R^3,
\end{equation}
where
\begin{equation}\label{def:Jac}
	M=\|\nabla A(t,x)\|_{L^\infty(\R^4)}.
\end{equation}

Next, since $\partial_t A\not\equiv0$, there exists $b\in\R^3$ such that the function $A(\cdot,b):[0,T]\to\R$ is non-constant. Recalling the notation \eqref{eq:decomposition}, let us take the primitive 
\begin{equation}\label{def:g(t)}
	g(t)=\int_0^t \tilde A(s,b) ds,\ \text{for every }t\in\R.
\end{equation}
Obviously we have that $g\in\boW$ with non-trivial derivative $\dot{g}\not\equiv 0$. In particular,
\begin{equation}\label{HypLemma}
   A(t,b)=\dot{g}(t),\hbox{ for all }t\in[0,T].
\end{equation}
On the other hand, we also define 
\begin{equation}\label{def:p-eps}
  p=-\varepsilon \tilde g + b,\quad\hbox{for all }\varepsilon>0,
\end{equation}
where $\varepsilon$ can be chosen sufficiently small to ensure that $p\in\boK$. Then,
\begin{align*}
    \int_0^T \dpp\cdot A(t,p) dt=\int_0^T\dpp\cdot\left(A(t,p)-A(t,b)+\dot{g}\right) dt=\int_0^T\dpp\cdot\left(A(t,p)-A(t,b)\right) dt-\varepsilon\|\dot{g}\|_2^2.
\end{align*}
Moreover, using the Lipschitz condition \eqref{eq:Lip}-\eqref{def:Jac}, we proceed as follows:
    \begin{align*}
        \Bigg|\int_0^T&\dpp\cdot(A(t,p) -A(t,b))dt\Bigg| = \varepsilon\left|\int_0^T \dot{g}\cdot\left(A(t,p)-A(t,b)\right)\ dt\right|\leq \varepsilon M\int_0^T |\dot{g}||p-b|dt
        \\
         &=\varepsilon^2M\int_0^T |\dot{g}||\tilde g|dt \leq \varepsilon^2M\left(\int_0^T |\tilde g|^2dt\right)^\frac12 \left(\int_0^T|\dot{g}|^2dt\right)^\frac12\leq \frac{\varepsilon^2MT}{\pi}\|\dot{g}\|_2^2,
    \end{align*}
where the last step follows directly from the Poincaré-Wirtinger inequality \eqref{eq:poincare}.

Finally, since
\begin{equation}\label{def:relativistic-ineq}
    1-\sqrt{1-s^2}\leq s^2,\quad \text{for all }s\in[0,1],
\end{equation}
    we deduce that
    \begin{equation}\label{ineq:epsilon1}
        \boI_0(p)\leq-\varepsilon\|\dot{g}\|_2^2   \left(1-\varepsilon\frac{\pi+MT}{\pi}\right).
    \end{equation}
  Hence, we obtain that $\boI_0(p)<0$ for all sufficiently small $\varepsilon>0$.
\end{proof}

\section{Extensions to \texorpdfstring{$\Phi\not\equiv 0$}{} with dominant vector potential}\label{Section-Phi}
Similarly to Theorem~\ref{MainTheorem1}, in this section we prove both Theorem~\ref{MainTheorem2} and Theorem \ref{MainTheorem3} as consequences of minimizing the associated action functional $\boI$ from \eqref{eq:functional} via the application of Theorem~\ref{thm:PLA} at each case. Moreover, as commented in Section~\ref{SectionMinimum}, the existence of an uniform bound $\rho\in(0,1)$ for the derivative of the minimizer when $\boS=\emptyset$ follows directly from the Lemma~\ref{lemma:uniformbound} in the Appendix.

In what follows we fix a compact subset $\boS\subset\R^3$. Let us also recall the function $\varphi:\boS^c\to\R$ defined in \eqref{def:varphi}, i.e.
\begin{equation}
    \varphi(x)=\int_0^T\Phi(t,x)dt,\quad\hbox{with }\Phi\in\boP,
\end{equation}
which is non-positive. 

After these remarks, Theorem~\ref{MainTheorem2} follows directly from the next result.
\begin{proposition}\label{MainTheorem2'}
    Let $(A,\Phi)\in\AP$ satisfy conditions \eqref{cond:varphi0}, \eqref{cond:varphi1} and \eqref{cond:varphi2}, for some sequence $\{b_n\}\subset\R^3$ with $\lim\limits_{n\to\infty}|b_n|=\infty$. Then there exists $q\in\boK_\Lambda$ such that $\boI(q)<0$.
\end{proposition}

\begin{proof}
    Let $\{b_n\}\subset\boS^c$ be the sequence with $\lim_{n\to\infty}|b_n|=\infty$ given by the statement. Similarly to \eqref{def:g(t)}, let us define the periodic function $g_n\in\boW$ as follows
\begin{equation}
    g_n(t)=\int_0^t\tilde A(s,b_n)ds,\quad \text{for all }n\in\mathbb{N}.
\end{equation}
Recall that $A(t,b_n)$ is non-constant for all $n\in\mathbb{N}$, and so
\begin{equation}
    \dot g_n(t)=\tilde A(t,b_n)\not\equiv 0,\quad\hbox{for all }n\in\mathbb{N}.
\end{equation}
Additionally, for $\varepsilon>0$, let $p_n$ be defined analogous to \eqref{def:p-eps}:
\begin{equation}
    p_n=-\varepsilon\tilde g_n + b_n,\quad\hbox{for all }n\in\mathbb{N}.
\end{equation} 
Since $\tilde A(t,b_n)\in\tilde\boW$, for all $n\in\mathbb{N}$, it follows that
\begin{equation}
    \max\limits_{t\in[0,T]}|\tilde A(t,b_n)|\leq T\|\partial_t A(t,x)\|_{L^\infty(\R^4)}\leq TM,\quad\hbox{for all }n\in\mathbb{N},
\end{equation}
where $M$ is given in \eqref{def:Jac}. Therefore, the sequence of norms $\|\dot g_n\|_\infty$ is uniformly bounded, and $\varepsilon>0$ can be chosen sufficiently small so that $p_n\in\boK$ for all $n\in\mathbb{N}$. Consequently, proceeding as in the proof of Lemma~\ref{lemma:negative}, we obtain the inequality
    \begin{equation}\label{ineq:epsilon2}
        \boI(p_n)\leq-\varepsilon\|\dot{g}_n\|_2^2   \left(1-\varepsilon\frac{\pi+MT}{\pi}\right)-\int_0^T\Phi(t,p_n)dt,\quad\hbox{for all }n\in\mathbb{N},
    \end{equation}
    which trivially reduces to \eqref{ineq:epsilon1} when $\Phi \equiv 0$. Hence, we now focus on estimating the term involving $\Phi$. First, notice that, for any $s\in [0,1]$, we have 
\begin{equation*}
    |s p_n(t)+(1-s)b_n|\geq |b_n|-\varepsilon| \tilde g_n|,\quad\hbox{for all }n\in\mathbb{N}. 
\end{equation*}
Then, recalling that 
\begin{equation}
    \|\tilde g_n\|_\infty\leq T\|\dot g_n\|_\infty,\quad\hbox{and}\quad\varepsilon \|\dot g_n\|_\infty<1,\quad\hbox{for all }n\in\mathbb{N},
\end{equation}it follows that
\begin{equation}\label{ineq:tvm}
    |s p_n(t)+(1-s)b_n|\geq |b_n|-T.
\end{equation}
    In particular, for all $n$ large enough, independent of $t$, every point in the segment joining $p_n(t)$ to $b_n$ belongs to $\boS^c$. Thus, by the mean value theorem there exists $s_n(t)\in [0,1]$ such that
\begin{align*}
    \left|\int_0^T(\Phi(t,p_n)-\Phi(t,b_n))dt\right|&\leq \varepsilon\int_0^T|\tilde g_n||\nabla\Phi(t,s_n(t)p_n+(1-s_n(t))b_n)|dt
    \\
    &\leq\frac{\varepsilon T}{\pi}\|\dot g_n\|_2\left(\int_0^T \left|\nabla\Phi(t,s_n(t)p_n+(1-s_n(t))b_n)\right|^2dt\right)^\frac12,
\end{align*}
where \eqref{eq:poincare} is used in the last step. As a consequence, recalling \eqref{ineq:tvm}, we deduce
\begin{equation*}
    \left|\int_0^T(\Phi(t,p_n)-\Phi(t,b_n))dt\right|\leq \frac{\varepsilon T^2}{\pi}\|\dot g_n\|_2\max\{|\nabla\Phi(t,x)|:\, t\in [0,T],\, |x|\geq|b_n|-T\}.
\end{equation*}
Therefore, condition \eqref{cond:varphi2} allows us to take $n$ sufficiently large so that
\begin{equation*}
    \left|\int_0^T(\Phi(t,p_n)-\Phi(t,b_n))dt\right|\leq \varepsilon^2\|\dot g_n\|_2^2.
\end{equation*}
Substituting this into \eqref{ineq:epsilon2} yields
\begin{equation}\label{ineq:epsilon2'}
        \boI(p_n) \leq-\varepsilon\|\dot{g}_n\|_2^2   \left(1-\varepsilon\frac{2\pi+MT}{\pi}\right)- \int_0^T\Phi(t,b_n)dt,
    \end{equation}
    for all $n\in\mathbb{N}$ sufficiently large. In addition, since $\int_0^T\Phi(t,b_n)dt=\varphi(b_n)$, condition \eqref{cond:varphi1} implies that
\begin{equation*}
        \boI(p_n) \leq-\varepsilon\|\dot{g}_n\|_2^2   \left(1-\varepsilon\frac{3\pi+MT}{\pi}\right),
    \end{equation*}
    for all $n \in \mathbb{N}$ sufficiently large. Therefore, taking $\varepsilon < \left(3 + MT\pi^{-1}\right)^{-1}$, we obtain that $\boI(p_n) < 0$, for all $n \in \mathbb{N}$ sufficiently large.
\end{proof}

We conclude this section by proving the existence of negative values of the action functional within the framework of Theorem~\ref{MainTheorem3}. The proof in this case differs significantly from that of Proposition~\ref{MainTheorem2'}, as it relies on a negative gradient flow. We stress that Lemma~\ref{lemma:negative} could also have been established using this alternative approach. The precise statement is as follows.

\begin{proposition}\label{MainTheorem3'}
	Let $(A,\Phi)\in\AP$ satisfy condition \eqref{cond:nonequilibrium-ElecField}. Then there exists $q\in\boK_\Lambda$ such that ~$\boI(q)<0$. 
\end{proposition}
\begin{proof}
Fix $(A,\Phi)\in\AP$ as in the statement, and take $b\in\boS^c$ according to hypothesis \eqref{cond:nonequilibrium-ElecField}, that is,
\begin{equation}
    \varphi(b)=0,\quad\hbox{and}\quad\partial_t A(\cdot,b)+\nabla\Phi(\cdot,b)\not\equiv0.
\end{equation}
Note that, trivially, $\boI(b)=-\varphi(b)=0$. Moreover, since $\boI(q)$ is differentiable for all $q\in\R^3$, it is straightforward to verify that
\begin{equation*}
    \boI'(q)[\psi]=\boF'(q)[\psi]=-\int_0^T\psi\left(\partial_t A(t,q)+\nabla\Phi(t,q)\right)dt,\quad\hbox{for all }q\in\R^3, \psi\in\boW.
\end{equation*}
Hence, $\boI'(b)\neq0$.

Let us consider the open subset of $\boW$ defined by
\[\Omega=\{q\in\Lambda:\, \|\dq\|_\infty<1,\, \boI'(q)\not=0\},\]
so that $b\in\Omega$. It is well-known (see, e.g., \cite[Lemma 2.2]{Willem}) that there exists a locally Lipschitz pseudogradient field $V:\Omega\to \boW$ for $\boI$ in $\Omega$. That is, for any $q\in\Omega$, one has
\begin{align}
    &\|V(q)\|_\boW\leq 2\|\boI'(q)\|_{\boW^*},
    \\
    &\boI'(q)[V(q)]\geq \|\boI'(q)\|^2_{\boW^*},
\end{align}
where $\boW^*$ denotes the dual space of $\boW$. Consequently, the initial value problem
\begin{equation}\label{eq:ivproblem}
	\frac{dx}{d\tau}(\tau)=-V(x(\tau)),\quad x(0)=b,
\end{equation}
has a unique solution $x\in\boC^1([0,\delta);\Omega)$, for some $\delta\ll1$; see, for instance, \cite{Deim} for details. Furthermore, the chain rule and the properties of the pseudogradient field yield
\begin{equation}\label{eq:derivative-1}
	\frac{d}{d\tau}(\boI(x(\tau))= \boI'(x(\tau))\left[\frac{dx}{d\tau}(\tau)\right]= -\boI'(x(\tau))\left[V(x(\tau))\right]\leq -\|\boI'(x(\tau))\|_{\boW^*}^2,\quad\text{for all }\tau\in [0,\delta).
\end{equation}
Hence the function $\tau\mapsto \boI(x(\tau))$ is decreasing in $(0,\delta_0)$ for some $\delta_0\in [0,\delta)$, and we conclude that
\begin{equation}
	\boI(x(\tau))<\boI(x(0))=\boI(b)= 0,\quad\text{for all }\tau\in (0,\delta_0),
\end{equation}
which completes the proof.
\end{proof}

\section{Conclusive remarks}\label{Section-conclusions}

We may rephrase Theorem~\ref{thm:PLA} by stating that, for any admissible pair $(A,\Phi)\in\AP$ for which $\boI$ attains negative values, there exists a non-constant periodic solution to the Lorentz \eqref{eq:LFE}. Taking this as the reference minimization result,  Theorem~\ref{MainTheorem1}, Theorem~\ref{MainTheorem2} and Theorem~\ref{MainTheorem3} provide sufficient conditions ensuring that $\boI$ does attain negative values. We conclude this work by analyzing closely the extent of necessity of the conditions just mentioned.

To begin with, we recall that the decay conditions \eqref{cond:decayA} and \eqref{cond:decayphi} are physically motivated and, more specifically, they lead to the uniqueness of solution to the Maxwell equations. However, from a purely mathematical point of view, these decay assumptions are not strictly necessary  to have action minimizers, as shown in \cite{ABT}. 

On the other hand, recall that an admissible pair of potentials generates a non-zero electric field, which provides the convenient framework for studying relativistic dynamics. Even though the absence of electric field is not an obstruction itself for the existence of action minimizers, it is not difficult to find examples in the magnetostatic regime such that $\boI$ is non-negative everywhere, thus preventing the application of Theorem~\ref{thm:PLA}. More concretely, let $\Phi\equiv 0$, and let $A\in\boC^1(\R^3;\R^3)$ satisfy \eqref{statement:vanishingA} and $\|\nabla A\|_{\infty} <\frac{\pi}{2T}$. Notice that $(A,\Phi)\not\in\AP$ due to the fact that $\partial_t A + \nabla\Phi\equiv 0$. Then, Poincaré--Wirtinger inequality \eqref{eq:poincare} implies that
    \begin{align*}
        \boI(q)&=\Psi(q) + \int_0^T \dq\cdot (A(q)-A(\bar q))dt\geq\frac12\int_0^T|\dq|^2 dt - \frac{\pi}{2T}\int_0^T |\dq||\tilde q| dt
        \\
        &\geq \frac12\int_0^T|\dq|^2 dt - \frac{\pi}{2T}\left(\int_0^T |\dq|^2 dt\right)^\frac12\left(\int_0^T|\tilde q|^2 dt\right)^\frac12\geq 0,\quad\hbox{for all }q\in\boK,
    \end{align*}
and, since $\boI(b)=0$ for all $b\in\R^3$, it follows that $\min\limits_{q\in\boW}\boI(q)=\boI(b)=0$. Consequently, Theorem~\ref{thm:PLA} fails to be applicable in this situation, as every constant function is a global minimizer of the functional.

Next, we point out that the nature of the singularities allowed for the potentials in $\AP$ must be carefully constrained in order to expect the existence of a global minimizer of the action functional. Indeed, if $\varphi$, as defined in \eqref{def:varphi}, were unbounded from above, then $\boI$ would be unbounded from below, so that no global minimum could exist. 

Even more interesting is the possibility of having singularities in the vector potential. For instance, assume for simplicity that $\Phi\equiv 0$, and let $A\in\boC^1(\R\times(\R^3\setminus\{0\});\R)$ be $T$-periodic. Assume in addition that there exists $\delta>0$ such that
\[\partial_t A(\cdot,b)\not\equiv 0,\quad\text{ for all }b\in \R^3\hbox{ with }|b|\in(0,\delta)\]
and, furthermore,
\begin{equation}\label{div_oscillations}
\lim_{b\to 0}\frac{\|\dot g\|^2_2}{\|\dot g\|_\infty}=\infty,\quad\text{where}\ \dot g(t)=\tilde A(t,b).
\end{equation}
It is clear that, in such a case,
\[\frac{\|\dot g\|^2_2}{\|\dot g\|_\infty}\leq \sqrt{T}\|\dot g\|_2\leq T\|\dot g\|_\infty\to \infty,\text{ as }b\to 0.\]
Therefore, condition \eqref{div_oscillations} may be interpreted as a strong divergence of the oscillations of $A(t,b)$ as $b$ approaches the singularity. Arguing as in the proof of Lemma~\ref{lemma:negative}, we derive that the point $p=b-\varepsilon\tilde g$ defined in \eqref{def:p-eps} 
satisfies the inequality
\[\boI(p)\leq -\varepsilon \|\dot g\|_2^2\left(1-\varepsilon\frac{\pi+MT}{\pi}\right),\]
for every $\varepsilon>0$ small enough so that
$\varepsilon\|\dot g\|_\infty\leq 1.$ Therefore, taking $\varepsilon=1/\|\dot g\|_\infty$, we deduce
\[\boI(p)\leq -\frac{\|\dot g\|_2^2}{\|\dot g\|_\infty}\left(1-\frac{\pi+MT}{\pi \|\dot g\|_\infty}\right)\to-\infty,\quad\text{as }b\to 0.\]
We have thus obtained a curve $b\mapsto p$ along which $\boI$ is always negative and diverges to $-\infty$, demonstrating that global minimizers cannot exist in this setting. Therefore, in the absence of a scalar potential, singular vector potentials may lead to a breakdown of the minimization framework. 
\\\\
{\bf Acknowledgments}: M. Garzón is partially supported by the Spanish MCIN/AEI grant JDC2022-049490-I and project PID2021-128418NA-I00.
\\\\
{\bf Data Availability Statement:} Data sharing is not applicable to this article as no datasets were generated or analyzed during the current study.

\appendix
\section{Appendix}\label{appendix-lemma}
Here we prove that the velocity of any periodic solution to the Lorentz force equation is uniformly bounded in the case of a continuous electromagnetic field. In our framework, this corresponds with the choice of $\boS=\emptyset$. Consequently, for any $(A,\Phi)\in\AP$ in this situation, the critical points of the functional $\boI$ are contained in a subset of $\boK$ that is compactly embedded on it. Hence all such trajectories remain at a positive distance from the boundary $\partial K$, which represents the set of motions approaching the speed of light in vacuum. The result is the following.
\begin{lemma}\label{lemma:uniformbound}
	For any $(A,\Phi)\in\AP$ there exists $\rho\in(0,1)$ such that 
	\begin{equation}\label{def:setKo}
		\|\dq\|_\infty< \rho,
	\end{equation}
    for every periodic solution $q(t)$ of the Lorentz force equation \eqref{eq:LFE}.
\end{lemma}
\begin{proof}
    First, since $\boS=\emptyset$, for any $(A,\Phi)\in\AP$, it is satisfied that
    \begin{equation}
        C:=\max\left\{\|\partial_t A\|_{L^\infty(\R^4)}, \max\limits_{(t,x)\in\R^4}|\textnormal{det}(\nabla A(t,x))|, \|\nabla\Phi\|_{L^\infty(\R^4)}\right\}<\infty,
    \end{equation}which follows from the decay conditions \eqref{cond:decayA} and \eqref{cond:decayphi}, and the regularity of the potentials.
    
   On the other hand, by introducing the relativistic momentum $p=\tfrac{\dq}{\sqrt{1-|\dq|^2}}$, the Lorentz force equation \eqref{eq:LFE} admits the following Hamiltonian formulation:
\begin{equation}\label{eq:LFEHamiltonian}
        \dq=\dfrac{p}{\sqrt{1+|p|^2}}, \quad
        \dpp=-\nabla\Phi(t,q)-\partial_t A(t,q)+\dq\times \nabla\times A(t,q). 
\end{equation}
Then, given a periodic solution $(q,p):[0,T]\to\R^6$ of \eqref{eq:LFEHamiltonian}, necessarily there exists a time $t_0\in[0,T]$ such that $\dq(t_0)=p(t_0)=0$. As a result, by integrating the second equation in \eqref{eq:LFEHamiltonian}, we directly obtain that
\begin{align*}
    |p(t)|&=\left|\int_0^t\left[\nabla\Phi(s,q(s))+\partial_t A(s,q(s))-\dq(s)\times \nabla\times A(s,q(s))\right]ds\right|\leq TC, \quad\hbox{for all }t\in[0,T],
\end{align*}
and using this bound in the first equation of \eqref{eq:LFEHamiltonian}, we conclude that
\begin{equation}
    |\dq(t)|\leq \dfrac{TC}{\sqrt{1+T^2C^2}} =:\rho<1,\quad\hbox{for all }t\in[0,T],
\end{equation}
which completes the proof.
\end{proof}

\end{document}